\pdfoutput=1
\documentclass[11pt,a4paper]{article}
\usepackage[T1]{fontenc}
\usepackage[utf8]{inputenc}
\usepackage{lmodern}
\usepackage{amsmath,amssymb,amsthm}
\usepackage{graphicx}
\usepackage{booktabs}
\usepackage{microtype}
\usepackage[margin=2.6cm]{geometry}
\usepackage{xcolor}
\usepackage[colorlinks=true,linkcolor=black,citecolor=black,urlcolor=blue!50!black]{hyperref}
\usepackage{caption}
\newtheorem{theorem}{Theorem}
\newtheorem{proposition}[theorem]{Proposition}
\newtheorem{corollary}[theorem]{Corollary}
\newtheorem{lemma}[theorem]{Lemma}
\theoremstyle{definition}

\theoremstyle{remark}
\newtheorem{remark}[theorem]{Remark}

\newcommand{\R}{\mathbb{R}}
\newcommand{\Rot}{\operatorname{Rot}}
\newcommand{\vol}{\operatorname{vol}}
\newcommand{\inter}{\operatorname{int}}
\newcommand{\GCLO}{0.4354}
\newcommand{\GCHI}{0.4472}
\newcommand{\ALLMAXG}{0.3875}
\newcommand{\ALLMAXTH}{15^\circ\text{--}35^\circ}

\title{A Pythagoras Tree Grown from a Dodecahedron}
\author{Miguel Esteban P\'erez Ibarra\\[2pt]
  \small Escuela Superior de F\'isica y Matem\'aticas, Instituto Polit\'ecnico Nacional,
  Mexico City\\
  \small \texttt{supersluky@gmail.com}}
\date{October 2026}

\begin{document}
\maketitle

\begin{abstract}
The Pythagoras tree of A.\,E.~Bosman (1942) is drawn with squares in the plane
and, in the three-dimensional versions we know of, with cubes.  This article
documents a tree built instead from regular dodecahedra: each dodecahedron
carries five smaller copies, one on each edge of a single pentagonal ``crown''
face, scaled by $g\approx0.4346$ and twisted by $45^\circ$ about the edge.  The
construction first existed only as a pair of rendered images and two
undocumented C\texttt{++} programs.  We recover its exact parameters from the
original Blender scene, reformulate it as an iterated function system with
condensation, and prove with a short computer-assisted argument that it never
collides with itself: two of its dodecahedra overlap only when one is the
parent of the other.  A Pythagorean identity for the six face axes of the
dodecahedron singles out the scale $1/\sqrt5$ and the twist $\arccos(1/\sqrt5)$.
With these values each child rests flat against a face of its parent, the five
children together have exactly the parent's surface area, and the crown of the
tree is a fractal of Hausdorff dimension exactly~$2$.  Finally we ask whether
the tree can grow on all twelve faces at once.  A parity argument shows that
at the artwork's proportions it cannot, but smaller scale factors admit
certified clash-free all-faces trees.
\end{abstract}

\normalsize

\begin{figure}[t]
  \centering
  \includegraphics[width=0.495\linewidth]{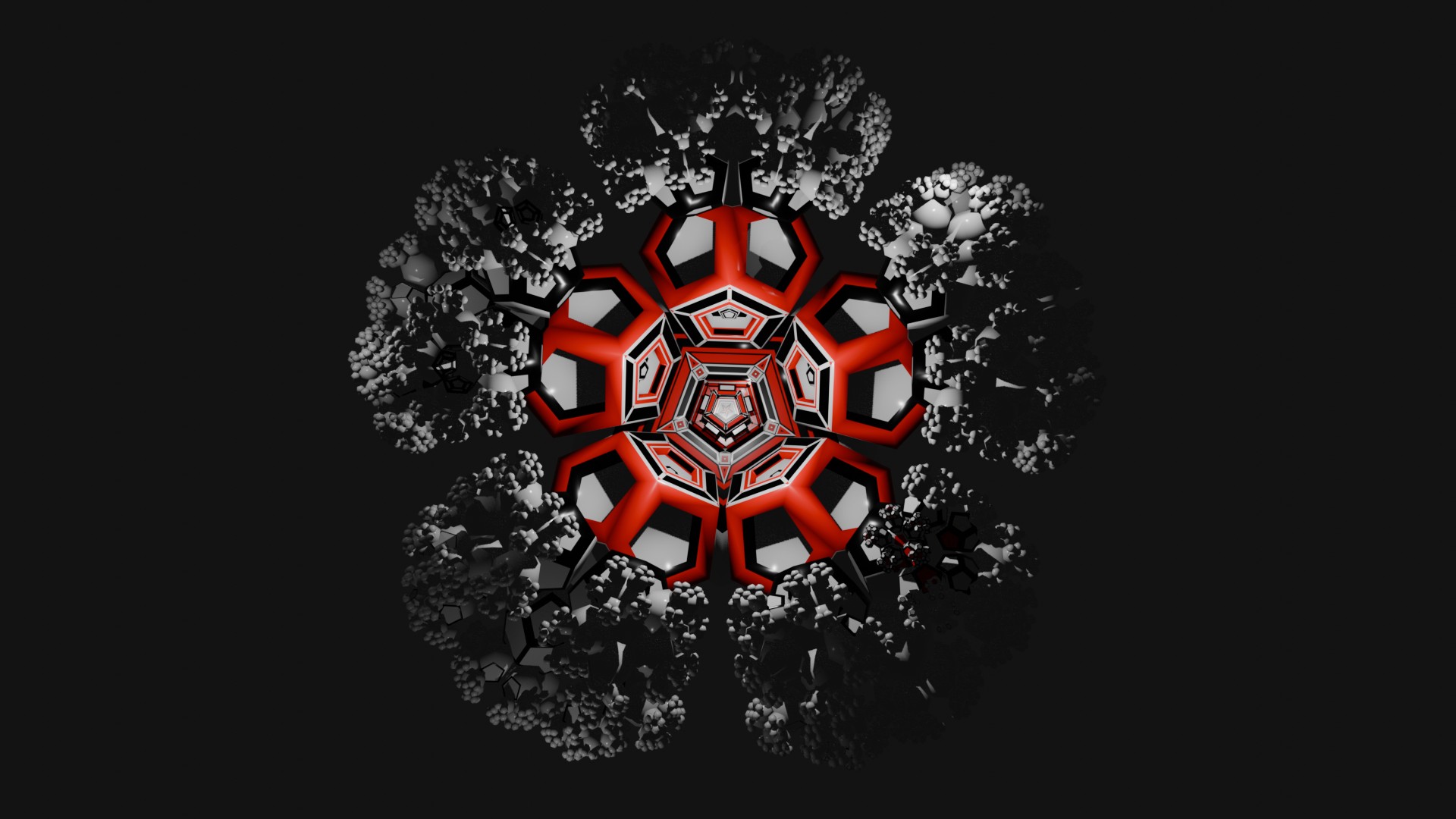}\hfill
  \includegraphics[width=0.495\linewidth]{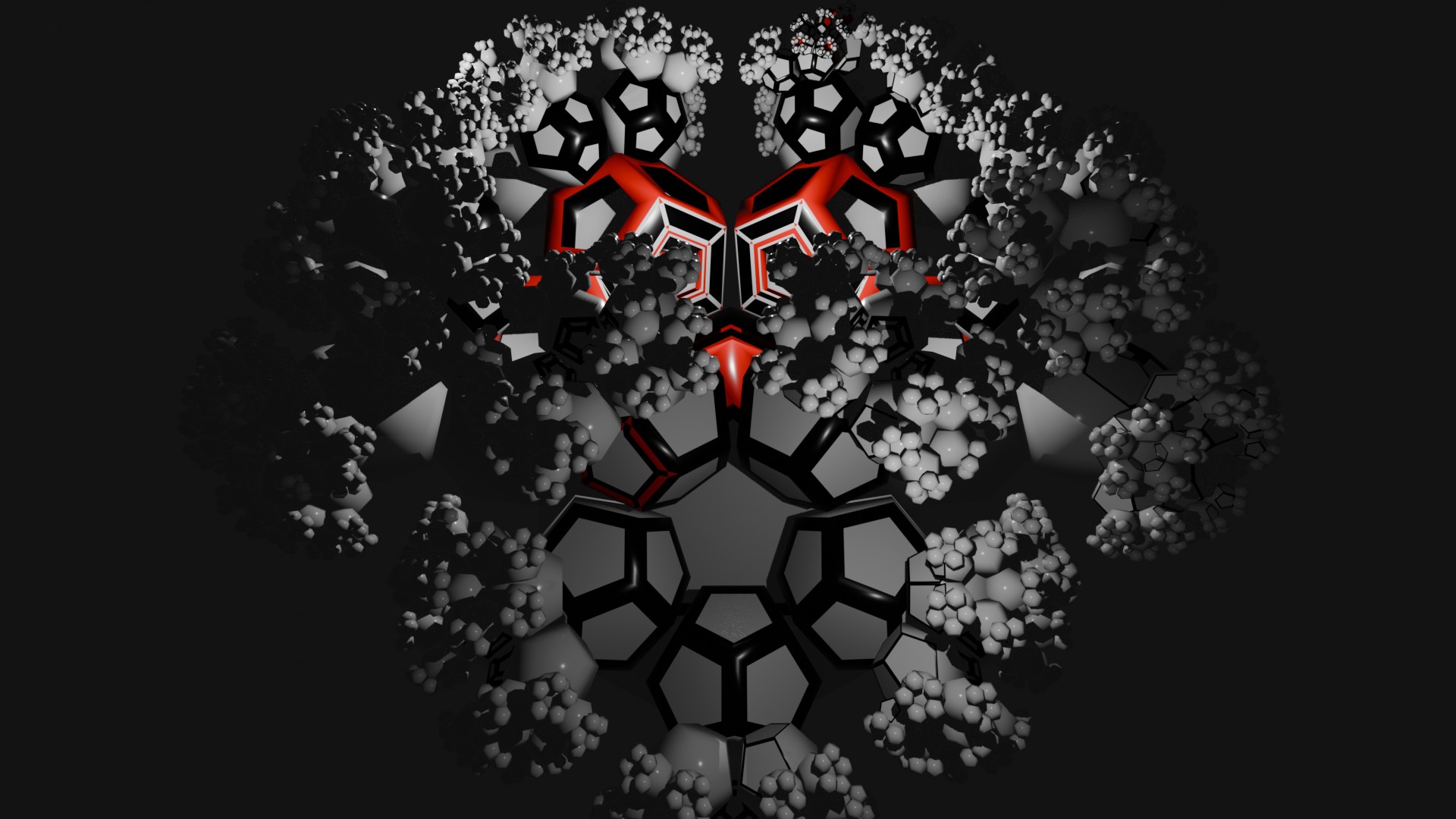}
  \caption{The artwork that started this article (rendered in Blender, March
  2026).  Left: looking down the axis of the crown face; right: the same tree
  from the side.  The red frames are the first three generations, remodelled by
  hand; everything else is the unedited output of the program described in
  Section~\ref{sec:construction}.}
  \label{fig:artwork}
\end{figure}

\section{Squares, cubes, and the missing dodecahedron}

Bosman's tree starts with a square.  On its top side sits a right triangle;
on each leg of the triangle sits a new square, and the construction repeats on
every new square.  If the triangle has angle $\alpha$, the two children are
scaled by $\cos\alpha$ and $\sin\alpha$, and the theorem of Pythagoras says
exactly that their areas add up to the area of the parent:
\begin{equation}\label{eq:pyth}
  \cos^2\alpha+\sin^2\alpha = 1 .
\end{equation}
The tree owes its name to this identity \cite{Lauwerier1991}.  It has been
generalised to free angles and non-square shapes
\cite{Beck2014,Munz2019,Ruta2024}.  In three dimensions the natural analogue
is built from cubes; the right triangle then becomes a tri-rectangular
tetrahedron and \eqref{eq:pyth} becomes de Gua's theorem on the squares of
its face areas \cite{ConantBeyer1974}.  Cube trees appear as 3-D prints and
even as model forests in turbulence simulations \cite{Schrottle2013}.

What does not seem to exist is a Pythagoras tree made of dodecahedra.  We
searched the scholarly indexes (arXiv full text, OpenAlex, Crossref, Semantic
Scholar), the Bridges archive, code platforms and the open web and found none;
every three-dimensional Pythagoras tree we located is built from cubes.  The
nearest relative is Whitehead's \emph{generalised Pythagorean lute}
\cite{Whitehead2022}, a string of dodecahedra scaled by the golden ratio, which
is a lute rather than a tree.  There is a good reason for the gap.  The joint
of a cube tree is a corner where three mutually perpendicular planes meet, and
any polygon with a corner at such a joint has a right angle there.  Regular
pentagons have $108^\circ$ corners, so no joint built from regular
dodecahedra can be right-angled, and the classical construction cannot be
copied literally.

A dodecahedral tree must therefore decide what to keep from Bosman.  The tree
in Figure~\ref{fig:artwork} keeps the recursion and the self-similarity and
replaces the right-angled joint by a twist.  This article is its
documentation: how it is built (Section~\ref{sec:construction}), how we
recovered it from the files that produced the images
(Section~\ref{sec:forensics}), why it does not collide with itself
(Section~\ref{sec:avoid}), which member of the family deserves to be called
Pythagorean (Section~\ref{sec:pythagorean}), and what happens when it is
allowed to grow from every face (Section~\ref{sec:allfaces}).

\section{The construction}\label{sec:construction}

\paragraph{The base.}  Let $\varphi=(1+\sqrt5)/2$.  The twenty points
\[
  (\pm\varphi^{-1},\pm\varphi^{-1},\pm\varphi^{-1}),\quad
  (0,\pm1,\pm\varphi^{-2}),\quad(\pm\varphi^{-2},0,\pm1),\quad(\pm1,\pm\varphi^{-2},0)
\]
are the vertices of a regular dodecahedron $B$ with edge $2\varphi^{-2}\approx
0.7639$, inradius $r_{\mathrm{in}}\approx0.8507$, circumradius
$\sqrt3/\varphi\approx1.0705$ and \emph{midradius}~$1$: every edge midpoint is
a unit vector.  We fix one face $F_0$, the \emph{crown}, with outward unit
normal $a_0=(0,1,\varphi)/\sqrt{2+\varphi}$, and list its edges cyclically
(counter-clockwise seen from outside) as $E_0,\dots,E_4$.  For each $k$ let
$m_k$ be the midpoint of $E_k$ (a unit vector) and $e_k$ the unit direction of
$E_k$ in the cyclic orientation.

\paragraph{The children.}  For a scale $g\in(0,1)$, an offset $d>1$ and a twist
angle $\theta$, define five similarities
\begin{equation}\label{eq:maps}
  f_k(x) = g\,\Rot(e_k,\theta)\,x + d\,m_k,\qquad k=0,\dots,4,
\end{equation}
where $\Rot(e,\theta)$ is the right-handed rotation by $\theta$ about the unit
vector $e$.  The child $f_k(B)$ is a copy of $B$ shrunk by $g$, turned by
$\theta$ about a line parallel to the edge $E_k$, and centred on the ray
through the midpoint of $E_k$ at distance $d$.  The artwork uses
\begin{equation}\label{eq:params}
  g = 0.36\cdot\tfrac{1+\sqrt2}{2}\approx 0.43456,\qquad
  d = 1.125\cdot\tfrac{1+\sqrt2}{2}\approx 1.35800,\qquad
  \theta = 45^\circ .
\end{equation}
(The factor $(1+\sqrt2)/2$ is called \texttt{gold} in the original program; the
golden ratio was probably intended, but these are the values that produced
Figure~\ref{fig:artwork}.)  The midpoint of $E_k$ lies at distance
$d-1\approx0.358$ from the child's centre, which is less than the child's
inradius $g\,r_{\mathrm{in}}\approx0.370$, so the parent's edge pokes into the
child.  This small overlap is the joint of the tree; it amounts to $0.16\%$ of
the child's volume.  Figure~\ref{fig:levels} shows the first generations.

\begin{figure}[t]
  \centering
  \includegraphics[width=0.245\linewidth]{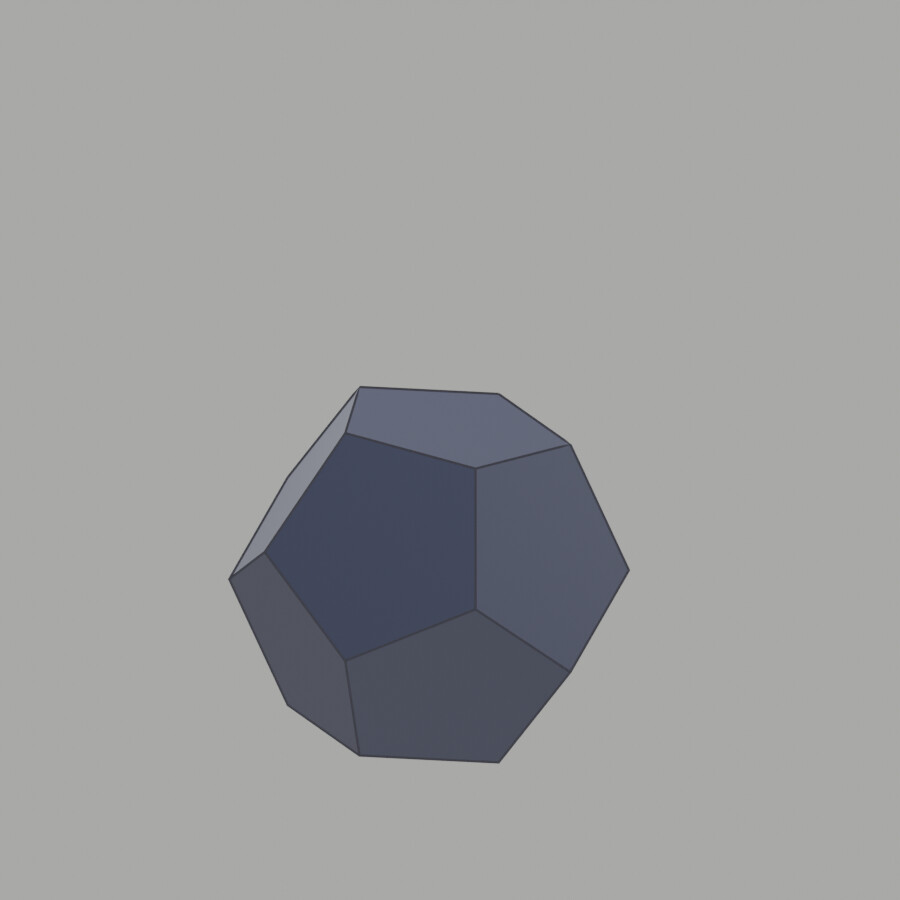}\hfill
  \includegraphics[width=0.245\linewidth]{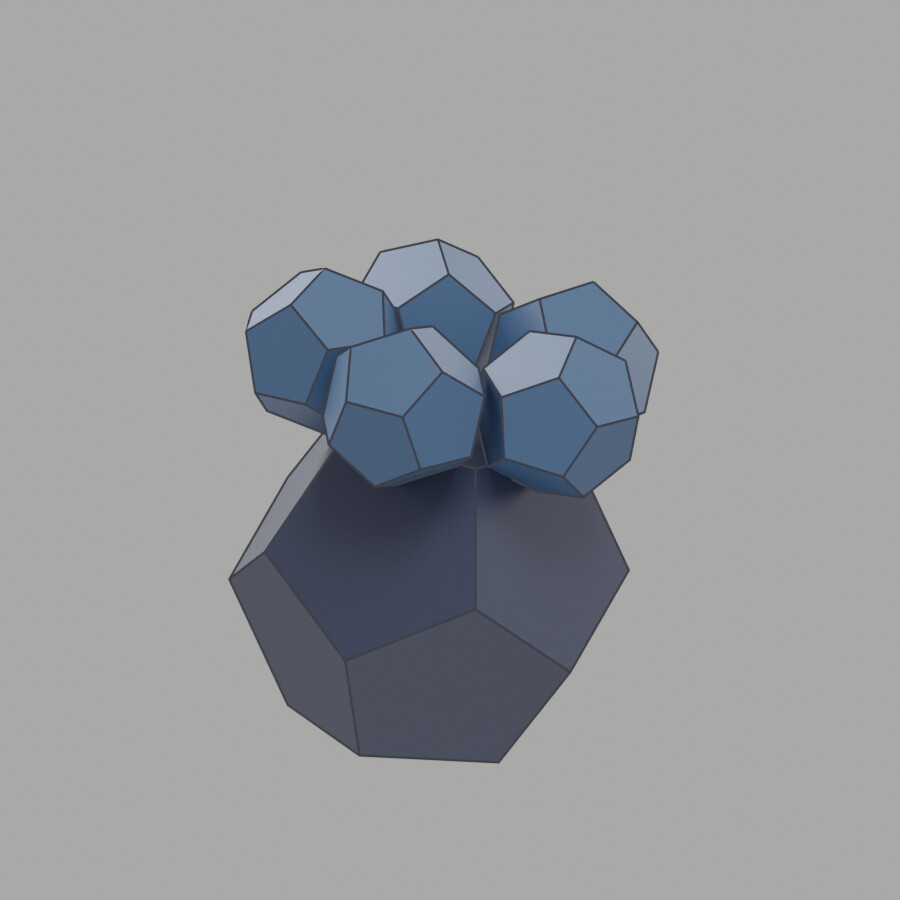}\hfill
  \includegraphics[width=0.245\linewidth]{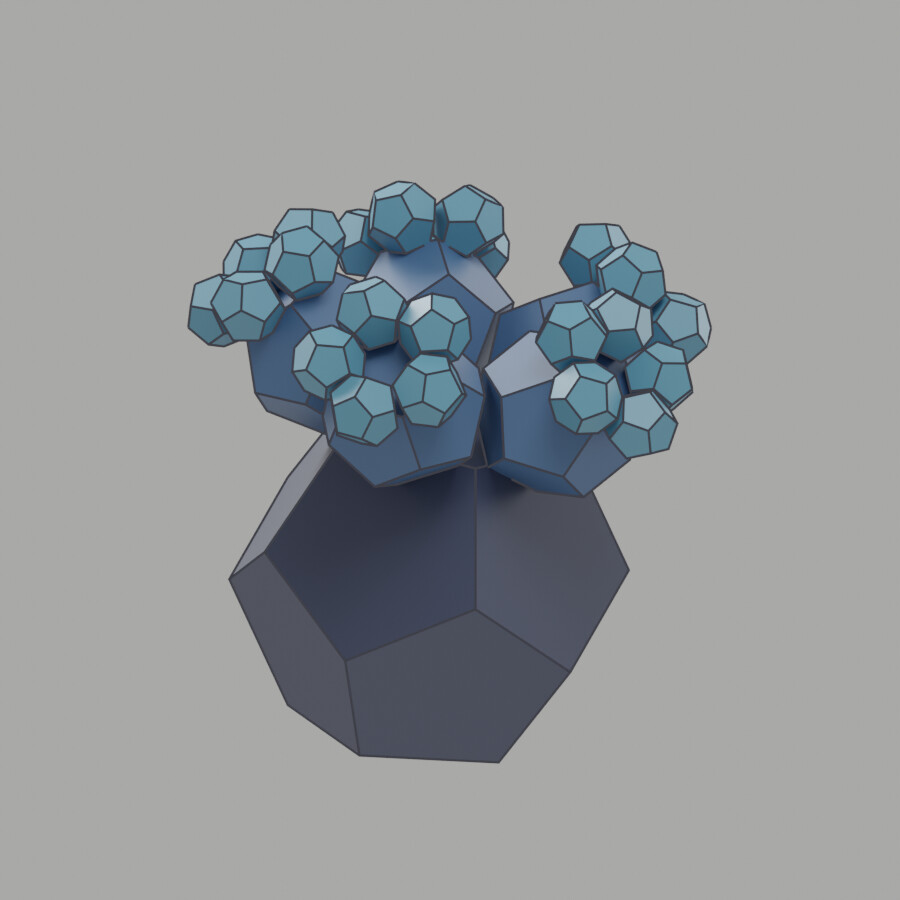}\hfill
  \includegraphics[width=0.245\linewidth]{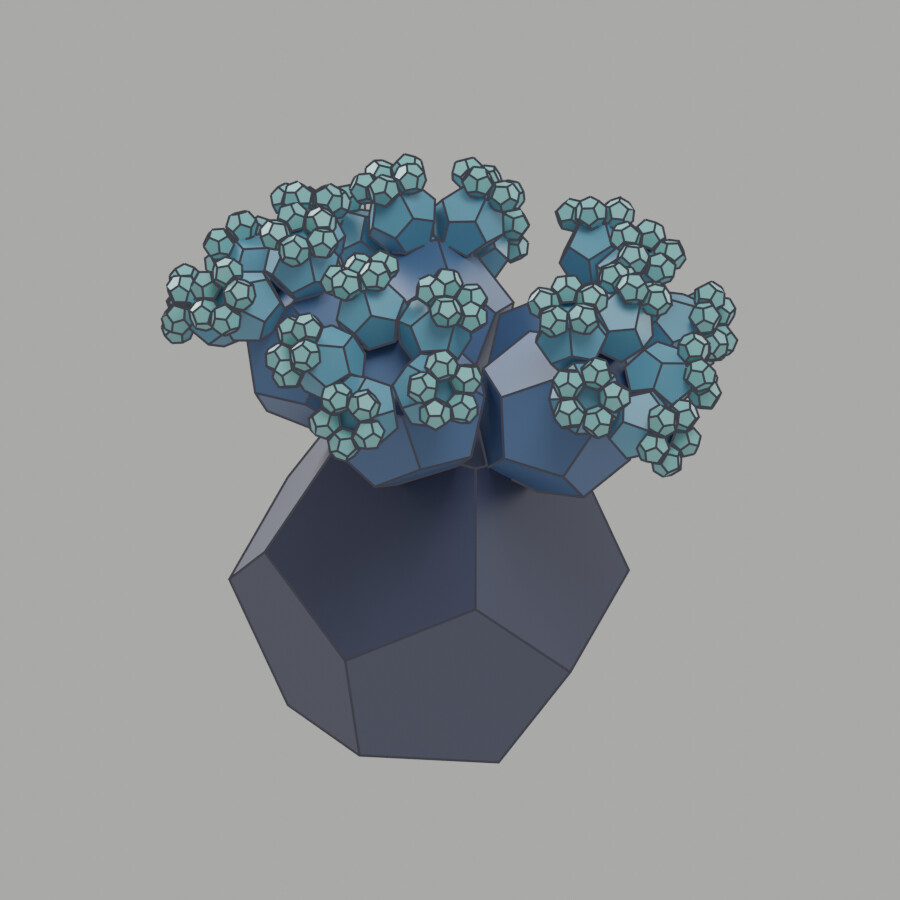}
  \caption{Generations $0$--$3$ of the tree with the parameters
  \eqref{eq:params}: $1$, $6$, $31$ and $156$ dodecahedra.  Each node sprouts
  five children on the edges of its own crown face.}
  \label{fig:levels}
\end{figure}

\paragraph{The tree as a fixed point.}  Let $\mathcal K$ be the space of
non-empty compact subsets of $\R^3$ with the Hausdorff metric.  The map
\[
  W(X) = B\ \cup\ f_0(X)\cup\dots\cup f_4(X)
\]
is a contraction of $\mathcal K$ with ratio $g$, because each $f_k$ is and the
fixed set $B$ does not increase Hausdorff distances.  By Banach's fixed point
theorem it has a unique fixed point, which we call the \emph{dodecahedral
Pythagoras tree}
\begin{equation}\label{eq:tree}
  T = B \cup f_0(T)\cup\dots\cup f_4(T).
\end{equation}
In Barnsley's terminology $T$ is the attractor of an iterated function system
\emph{with condensation} $B$ \cite{Barnsley1988,Hutchinson1981}.  Unwinding
\eqref{eq:tree}, $T$ is the closure of the union of the \emph{nodes}
$f_w(B)=f_{w_1}\circ\dots\circ f_{w_n}(B)$ over all finite words $w$ in the
alphabet $\{0,\dots,4\}$; the node $f_w(B)$ has depth $n$, scale $g^n$, and
children $f_{wk}(B)$ attached to the edges of its own crown $f_w(F_0)$.  The
depth-$n$ picture contains $(5^{n+1}-1)/4$ dodecahedra; Figure~\ref{fig:artwork}
shows depth $5$, that is, $3906$ of them.

\begin{remark}
The original program recurses on absolute coordinates, so the node it reaches
by choosing $k_1,\dots,k_n$ is $f_{k_n}\circ\dots\circ f_{k_1}(B)$, the composition
taken in the opposite order.  Both orders run over the same set of words, so
they draw the same picture, but only the order used above makes a node's
children sit on its own edges.
\end{remark}

\section{Recovering the artwork}\label{sec:forensics}

The images in Figure~\ref{fig:artwork} came with two C\texttt{++} programs that
write STL files, and the author no longer remembered which of the two had been
used.  We settled the question by checking each program against the actual
geometry.

Both programs place the children as in \eqref{eq:maps}, but they type the unit
vectors $m_k$ and $e_k$ with two decimals (for instance $(0.5, 0.31, 0.81)$ for
$(\tfrac12,\tfrac1{2\varphi},\tfrac\varphi2)$), build the twist by conjugating a
rotation about the $z$-axis with Euler rotations, and use a rotation-matrix
routine that takes degrees and whose $y$-rotation is the transpose of the usual
one.  We ported the programs to Python with every one of these quirks and
compiled the originals.  On a solid five-child build of depth $3$ ($156$
dodecahedra) the port and the compiled C\texttt{++} agree to $5\times10^{-7}$,
which is the print precision of the STL file.  One quirk turned out to matter
geometrically.  The program writes $-45^\circ$ for children $3$ and $4$, but
their axes are typed with the opposite orientation, so all five twists agree:
each is $+45^\circ$ about the cyclically oriented edge, as in~\eqref{eq:maps}.

The Blender scene behind both images holds a single mesh of $148{,}164$
polygons, $141{,}127$ of them triangles.  Splitting the triangles into
connected components gives $7{,}909$ pieces, of which $3{,}794$ are clean
regular dodecahedra (the rest are fragments and the polygons of the hand-built
frames).  Of these $3{,}794$, $2{,}435$ coincide with nodes of the depth-$5$
tree \eqref{eq:tree}--\eqref{eq:params} to within $0.5\%$ of their size.  The
other $1{,}359$ are still exactly regular, but have been moved by at most half
their own size (median $5\%$) and turned by at most $1.1^\circ$, which we read
as hand edits in Blender.  Generations $0$--$2$ were remodelled into the red
frames.  The second program, \texttt{DedocaedroA-copia.cpp}, passes radians
where degrees are expected, so its twists are only $0.785^\circ$, and it
matches none of the components.  The artwork is therefore
\texttt{DedocaedroTREEverCELLS.cpp} with all five children switched on, depth
$5$, and faces drawn solid.  (As it stands, that file enables only children $0$
and $1$ and draws every edge as a thin strut, the ``cells'' of its name.)
Replacing the typed two-decimal vectors by exact ones changes the maps by
$0.4\%$ and moves no vertex of the depth-$5$ tree by more than $0.007$; from now
on we use the exact version.

\begin{figure}[t]
  \centering
  \includegraphics[width=0.495\linewidth]{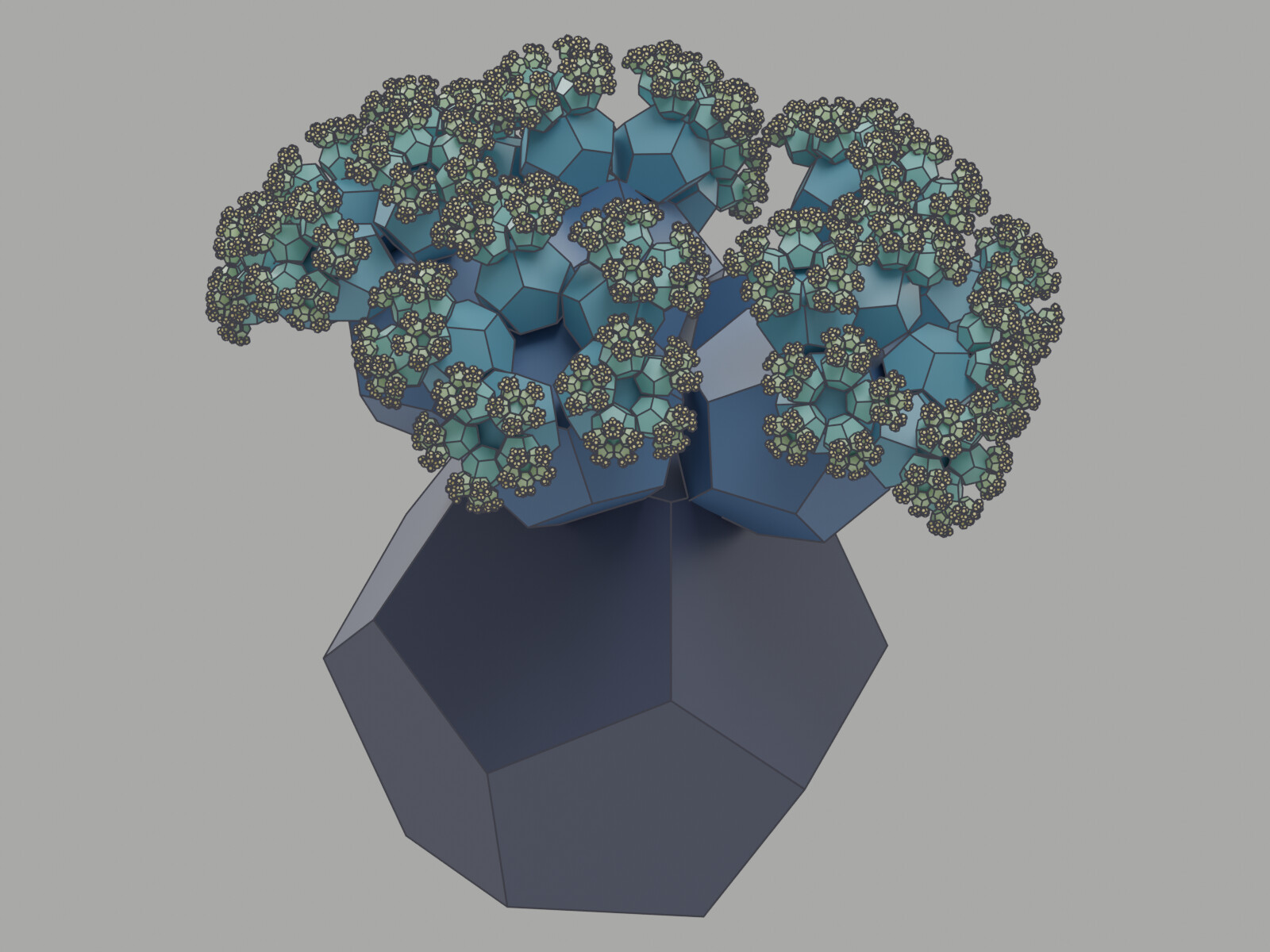}\hfill
  \includegraphics[width=0.495\linewidth]{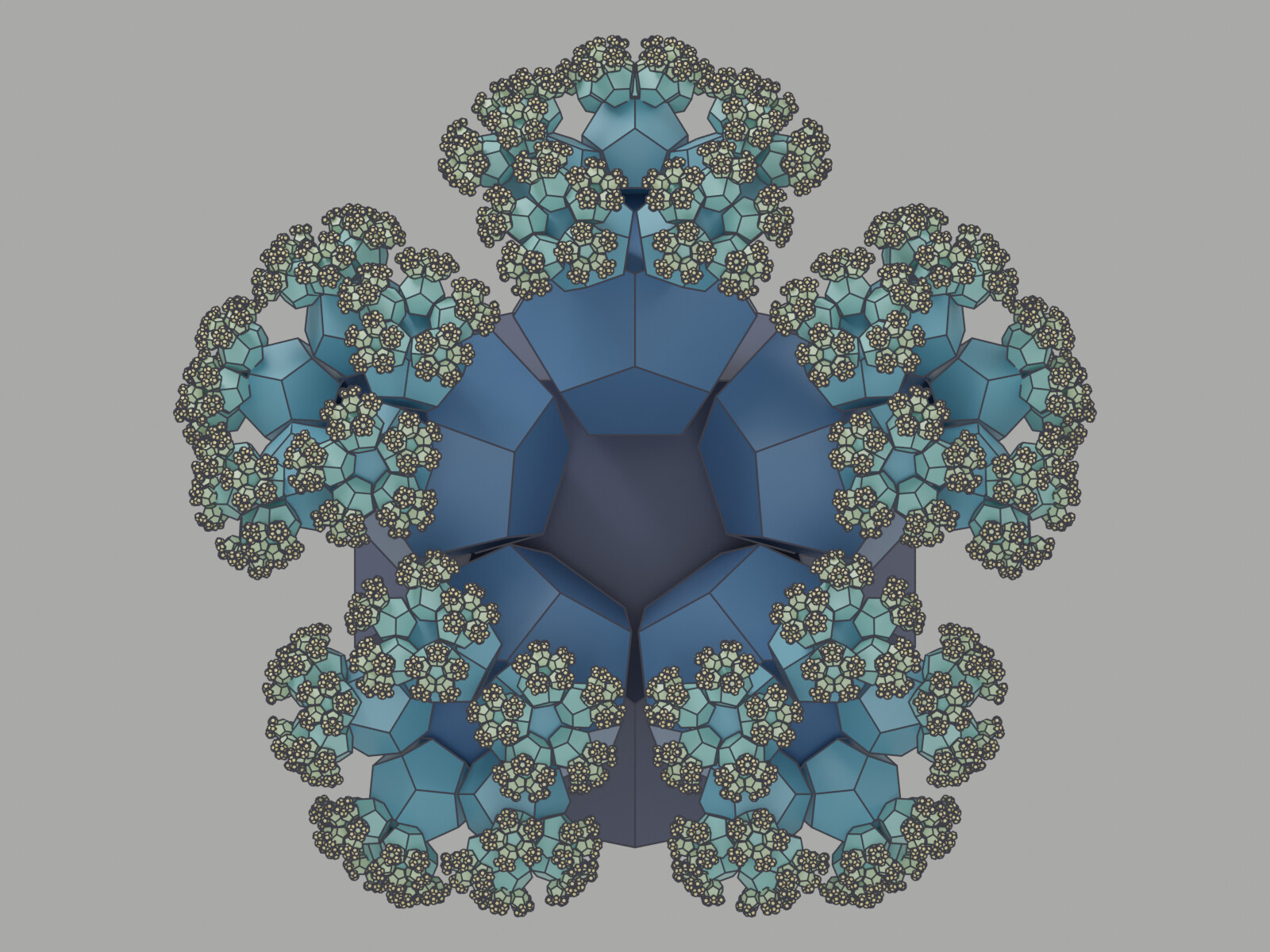}
  \caption{The artwork's tree regenerated from \eqref{eq:maps}--\eqref{eq:params},
  depth $5$, coloured by generation.}
  \label{fig:repro}
\end{figure}

\section{Symmetry and size}

The rotations of $B$ form the icosahedral group $I$ of order $60$; with the
reflections it becomes the full group $I_h$ of order $120$.  The tree keeps
only the part of it that fixes the crown.

\begin{proposition}\label{prop:sym}
For every $g,d,\theta$ the tree $T$ is invariant under the dihedral group
$D_5$ of order $10$ that fixes $F_0$: the rotations by multiples of $72^\circ$
about $a_0$, and the five reflections in planes through $a_0$ and an edge
midpoint $m_k$.  For $\theta=45^\circ$ no other rotation of $I$ preserves the
family $\{f_k\}$.
\end{proposition}

\begin{proof}
If $\sigma$ is an isometry with $\sigma(B)=B$ and $\{\sigma f_k\sigma^{-1}\}=\{f_k\}$,
then $\sigma(T)$ satisfies \eqref{eq:tree}, so $\sigma(T)=T$ by uniqueness.  A
rotation $\rho$ by $72^\circ$ about $a_0$ sends $(m_k,e_k)$ to
$(m_{k+1},e_{k+1})$, so $\rho f_k\rho^{-1}=f_{k+1}$.  A reflection $\mu$
fixing $F_0$ reverses the cyclic orientation, so it sends $e_k$ to $-e_{j}$
for the image edge $E_j$, and conjugation by a reflection reverses the sense
of a rotation; the two reversals cancel:
$\mu\Rot(e_k,\theta)\mu^{-1}=\Rot(-e_j,-\theta)=\Rot(e_j,\theta)$.  The last
claim we checked directly on the $60$ elements of $I$.
\end{proof}

Generation $n$ consists of $5^n$ dodecahedra of scale $g^n$.  Counting
volume and surface with multiplicity, the whole tree has
\[
  \frac{\vol(B)}{1-5g^3}\approx1.696\,\vol(B),\qquad
  \frac{\operatorname{area}(\partial B)}{1-5g^2}\approx 17.9\,\operatorname{area}(\partial B)
\]
for the artwork, and the surface series converges only because
$5g^2\approx0.944<1$.  The points that are limits of ever smaller nodes form
the \emph{crown set} $K$, the attractor of $\{f_0,\dots,f_4\}$ alone.  Its
similarity dimension $s$ solves Moran's equation $5g^s=1$, so
$s=\log5/\log(1/g)\approx1.931$.  This equals the Hausdorff dimension of $K$
as soon as the five pieces $f_k(K)$ are pairwise disjoint
\cite{Moran1946,Hutchinson1981,Falconer2014}, and the certificate behind
Theorem~\ref{thm:avoid} below keeps them a positive distance apart.

\section{The tree does not collide with itself}\label{sec:avoid}

Children overlap their parents by design.  We call any other pair of nodes
whose interiors meet a \emph{clash}.  A tree is \emph{self-avoiding} if it has
no clashes at any depth.

\begin{theorem}\label{thm:avoid}
With the parameters \eqref{eq:params}, the dodecahedral Pythagoras tree is
self-avoiding.  Two of its dodecahedra have intersecting interiors if and only
if one is the parent of the other.  The same holds for the two-decimal version
in the original program, and for its two-child variant.
\end{theorem}

The theorem covers every depth, while a computer can only look at finitely many
nodes.  Self-similarity bridges the gap.

\begin{lemma}\label{lem:reduction}
Let $Q$ be a convex polytope with $B\subseteq Q$ and $f_k(Q)\subseteq Q$ for all
$k$, and put $P=B\cup f_0(Q)\cup\dots\cup f_4(Q)$.  Suppose that
\begin{enumerate}
\item[(i)] for $j\ne k$, every convex piece of $f_j(P)$ has interior disjoint
from every convex piece of $f_k(P)$; and
\item[(ii)] $\inter B\cap f_kf_j(Q)=\emptyset$ for all $j,k$.
\end{enumerate}
Then $T$ is self-avoiding.
\end{lemma}

\begin{proof}
Since $W$ maps subsets of $Q$ into $Q$, the fixed point satisfies $T\subseteq Q$,
and then $T=W(T)\subseteq P$.  Take distinct nodes $f_u(B)$ and $f_v(B)$ and
strip their longest common prefix $w$; since $f_w$ is a similarity, this does
not change whether their interiors meet.  If one remaining word is empty and
the other has length $1$, the nodes are parent and child.  If one is empty and
the other is $kj\cdots$, the second node lies in $f_kf_j(T)\subseteq f_kf_j(Q)$
and (ii) separates it from $B$.  Otherwise the words begin with different
letters $j\neq k$, the nodes lie in $f_j(T)\subseteq f_j(P)$ and
$f_k(T)\subseteq f_k(P)$, and (i) separates them.
\end{proof}

\begin{proof}[Proof of Theorem~\ref{thm:avoid} (computer-assisted)]
We fix about $1{,}000$ unit directions $u_i$ (the facet normals of the convex
hull of the depth-$4$ tree, plus $400$ random directions) and look for a
polytope $Q=\{x: u_i\cdot x\le h_i\}$ that contains $B$ and its own images.
Starting from the hull of the depth-$4$ tree we iterate
$h_i\leftarrow\max\bigl(h_B(u_i),\max_k\,[\,u_i\cdot t_k+g\,h_Q(R_k^{\top}u_i)\,]\bigr)$,
where $f_k(x)=gR_kx+t_k$ and $h_Q$ is the support function of the current
$Q$, evaluated on its vertices.  The iteration settles to machine precision in
a few dozen steps.  We then enlarge every $h_i$ by $10^{-9}$ and verify
$B\subseteq Q$ and $f_k(Q)\subseteq Q$ directly on the vertices; the resulting
$Q$ has $1380$ vertices.  Conditions (i) and (ii) are disjointness statements
about pairs of convex polytopes.  Each is decided by a linear program that
finds the largest ball contained in both polytopes, where a negative radius
$-r$ means that the two polytopes can each be pushed outwards by $r$ before
they touch.  A negative optimum therefore certifies not only disjoint
interiors but a positive distance.  For \eqref{eq:params} the smallest margin
in (i) is $7.3\times10^{-3}$ and in (ii) it is $0.30$, both far above
floating-point error (the tree has diameter about $3.3$).  For the two-decimal
constants the margins are $5.8\times10^{-3}$ and $0.30$, and for the two-child
variant $1.1\times10^{-2}$ and $0.31$.  An independent check agrees: an exact
separating-axis test on all pairs of the depth-$5$ tree (after a
bounding-sphere filter) finds exactly the $3905$ parent--child joints and
nothing else.
\end{proof}

The twist, and its sense, are essential.  With $\theta=0^\circ$ or
$\theta=90^\circ$ (and the same $g,d$), adjacent children already overlap at
depth $1$, and by depth~$5$ there are $7615$ and $3905$ clashes respectively.
Twisting the other way, $\theta=-45^\circ$, gives $6740$ clashes by depth~$4$.
The first generation cannot tell $\theta$ from $\theta+180^\circ$, because an
edge direction is a two-fold axis of $B$, but the later generations can.
Figure~\ref{fig:phase} (top) maps the parameter plane: each point is checked to
depth~$4$ and, when clean, put through the certificate.  For scales
$g\ge0.2$, the range we scanned, self-avoiding crown trees exist only for
twists between roughly $-85^\circ$ and $140^\circ$; the largest certified
scale, $g=0.4625$, occurs near $\theta=65^\circ$.  The artwork sits close to
the edge of the safe region, which at $\theta=45^\circ$ ends at a critical
scale $g_c\in[\,\GCLO,\,\GCHI\,]$: the scale $\GCLO$ is certified, and at
$\GCHI$ a clash appears.  The artwork's $g=0.4346$ is certified as well.  The
upper value is $1/\sqrt5$, which will become important in the next section.

\begin{figure}[t]
  \centering
  \includegraphics[width=\linewidth]{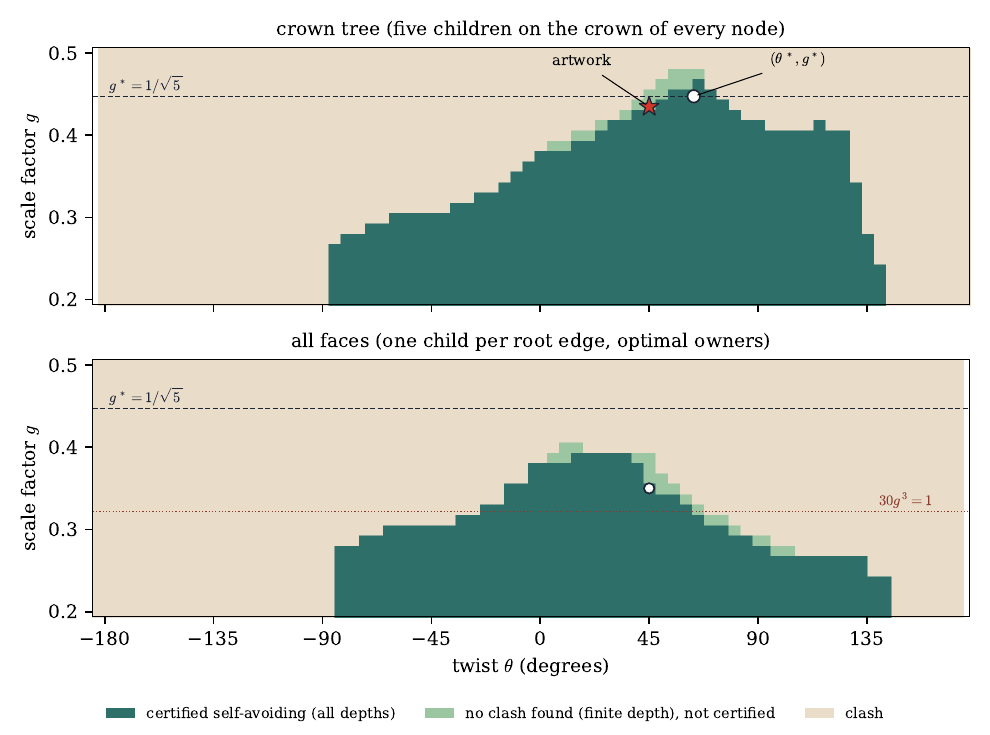}
  \caption{Where the tree is self-avoiding.  Twist $\theta$ against scale $g$,
  with offset $d=1+\kappa g$, $\kappa=(d_{\mathrm{art}}-1)/g_{\mathrm{art}}\approx0.824$, so that every
  child keeps the artwork's joint.  Top: the crown tree.  Bottom: the best
  all-faces tree of Section~\ref{sec:allfaces}.  The dotted line in the bottom panel
  is the volume bound of Proposition~\ref{prop:volume} for trees that sprout on
  all thirty edges at every node.  Twists of opposite sense give genuinely
  different trees, so the whole circle of twists is shown.}
  \label{fig:phase}
\end{figure}

\section{The Pythagorean dodecahedral tree}\label{sec:pythagorean}

Which choice of $g$ deserves the name Pythagoras?  In Bosman's tree the scale
of each child is the cosine of the angle between its growth direction and the
parent's, and \eqref{eq:pyth} says that these squared cosines sum to $1$.  The
dodecahedron has its own version of \eqref{eq:pyth}.  It is a classical fact
about the six face axes, which form an equiangular tight frame
\cite{LemmensSeidel1973,Waldron2018}; we include the short proof because the
tree will use it.

\begin{proposition}\label{prop:frame}
Let $a_1,\dots,a_6$ be unit vectors along the six axes through opposite face
centres of a regular dodecahedron.  Then for every $u\in\R^3$
\begin{equation}\label{eq:frame}
  \sum_{i=1}^{6}(u\cdot a_i)^2 = 2\,|u|^2 .
\end{equation}
In particular, if $F_0$ has normal $a_0$ and $a_1,\dots,a_5$ are the normals of
the five faces adjacent to it, then $\sum_{k=1}^5(a_0\cdot a_k)^2=1$ and each
term equals $1/5$.
\end{proposition}

\begin{proof}
The matrix $S=\sum_i a_ia_i^{\top}$ is symmetric.  It commutes with every
rotation in $I$, because $I$ permutes the six axes and the sign of $a_i$ does
not affect $a_ia_i^{\top}$.  Each eigenspace of $S$ is therefore invariant
under $I$.  But $I$ fixes no line in $\R^3$: a line kept by a rotation of order five
must be its axis, and $I$ contains five-fold rotations about different axes.
Hence it fixes no plane either (the normal of an invariant plane
would be an invariant line).  So $S$ has a single eigenspace, $S=\lambda\,
\mathrm{Id}$, and $3\lambda=\operatorname{tr}S=\sum|a_i|^2=6$.  For the second
statement, the six axes are equiangular, $|a_i\cdot a_j|=1/\sqrt5$ for
$i\ne j$, and $u=a_0$ gives $1+5\cdot\tfrac15=2$.
\end{proof}

For the cube the three face axes form an orthonormal frame and
\eqref{eq:frame} with constant $1$ is the theorem of Pythagoras in coordinates;
de Gua's theorem is its statement about areas \cite{ConantBeyer1974}.
Proposition~\ref{prop:frame} has the same reading: a planar region of area $A$
casts shadows on the six face planes of the dodecahedron whose squared areas
add up to $2A^2$.  For the crown face itself this is a dodecahedral analogue
of de Gua's theorem: the shadows of the five neighbouring faces on the plane
of $F_0$ have area $A/\sqrt5$ each, and their squares add up to $A^2$.

If the five children of a node grow along the five neighbouring face axes and
are scaled by the cosine of their tilt, which is the cosine rule of Bosman's
tree with the five axes of Proposition~\ref{prop:frame} in place of two
perpendicular legs, then their surface areas add up exactly to the parent's.
It turns out that the twist of \eqref{eq:maps} can produce exactly this.

\begin{proposition}\label{prop:kiss}
Let $g^*=1/\sqrt5$, $\theta^*=\arccos(1/\sqrt5)\approx63.43^\circ$ (the angle
between adjacent face axes) and $d^*=1+g^*$.  For each $k$ let $F'_k$ be the
face of $B$ that meets the crown $F_0$ along $E_k$, with normal $n_k$.  Then
the child $f_k(B)$
\begin{enumerate}
\item has its crown axis along $n_k$, so it grows along a neighbouring face
axis with scale $g^*=a_0\cdot n_k$;
\item lies on the far side of the plane of $F'_k$, its base face lying in that
plane; the base face covers a small triangle of $F'_k$ next to the edge $E_k$,
so child and parent touch along that triangle but their interiors are disjoint.
\end{enumerate}
\end{proposition}

\begin{proof}
The vectors $a_0$, $n_k$ and $m_k$ are all orthogonal to $e_k$, the angle from
$a_0$ to $n_k$ is $\theta^*$, and $m_k$ bisects it.  So $R=\Rot(e_k,\theta^*)$
maps $a_0$ to $n_k$, which is (1), and maps the normal $-a_0$ of the base face
(the face opposite the crown) to $-n_k$.  The base face of the child therefore
lies in the plane $\{x\cdot n_k=c\}$ with
$c = d^*\,m_k\cdot n_k - g^*\,r_{\mathrm{in}} = (1+g^*)r_{\mathrm{in}}-g^*r_{\mathrm{in}}=r_{\mathrm{in}}$,
using $m_k\cdot n_k=r_{\mathrm{in}}$ (the midpoint of $E_k$ lies on $F'_k$).
That plane is the plane of $F'_k$, which supports $B$, so the two solids lie on
opposite sides of it.  In that plane the centre of the child's base face lies
$g^*\rho$ beyond the line of $E_k$, where $\rho$ is the inradius of a face,
and the base face has a vertex pointing straight back at $E_k$.  Since the
circumradius $R_f$ of a face exceeds $\rho$, that vertex crosses the edge line
by $g^*(R_f-\rho)\approx0.0555$, while the next two vertices stay outside.
The contact region is therefore a triangle of area $\approx0.0042$, about
$2\%$ of the child's face.
\end{proof}

Moving the child any closer pushes the contact triangle into the parent, so
$d^*=1+g^*$ is the \emph{kissing distance}: the smallest offset at which child
and parent do not overlap.  (For convex bodies translated along a ray the
offsets that produce an overlap form an interval, so this distance is well
defined.)  The kissing distance is exactly $1+g$ for every $g$ at
$\theta\in\{0^\circ,\theta^*,90^\circ\}$, and by the mirror symmetry of
Proposition~\ref{prop:sym} also at $180^\circ-\theta$.  At $0^\circ$ and
$90^\circ$ the reason is different from the face contact at $\theta^*$: the
child then presents an edge midpoint to the parent's edge midpoint (a parallel
edge at $0^\circ$, a crossed one at $90^\circ$), and the plane
$\{x\cdot m_k=1\}$ separates the two solids once $d\ge1+g$.  At other twists
the ratio $(d-1)/g$ at kissing varies: numerically it is $0.869$, $0.874$ and
$1.015$ at $20^\circ$, $45^\circ$ and $80^\circ$.  The artwork's $0.824$ is
below the kissing value at $45^\circ$, which is the joint.

\begin{theorem}\label{thm:pythtree}
The tree $T^*$ defined by $(g^*,d^*,\theta^*)$ has the following properties.
\begin{enumerate}
\item \emph{(Pythagorean)} The five children of every node have, together,
exactly the surface area of their parent.  Every generation has the same total
surface area, and the total volume is $\vol(B)/(1-5^{-1/2})\approx1.81\,\vol(B)$.
\item \emph{(Packing)} Its dodecahedra have pairwise disjoint interiors.  Each
child touches its parent along a triangle and touches nothing else.
\item \emph{(Dimension two)} Its crown set $K^*$ has Hausdorff dimension
exactly $2$ and finite, positive two-dimensional Hausdorff measure.
\end{enumerate}
\end{theorem}

\begin{proof}
(1) is $5(g^*)^2=1$.  For (2), parent and child are separated by
Proposition~\ref{prop:kiss}, and all other pairs by the certificate of
Lemma~\ref{lem:reduction}, with margins $4.5\times10^{-2}$ in (i) and $0.34$ in
(ii).  For (3), $K^*\subseteq Q$ because $Q$ is invariant under the maps, so
$f_j(K^*)\subseteq f_j(P)$, and the positive margins in (i) show that the
compact sets $f_j(K^*)$ are pairwise disjoint.  This is the strong separation
condition, and the dimension is the solution of $5\,(1/\sqrt5)^s=1$, that is,
$s=2$ \cite{Hutchinson1981,Falconer2014}.
\end{proof}

So the dodecahedron supports a Pythagoras tree after all.  The right-angled
joint is replaced by five equiangular directions, and $\cos^2\alpha+\sin^2\alpha=1$
by $5\cdot(1/\sqrt5)^2=1$.  Figure~\ref{fig:pyth} shows $T^*$.  With the
artwork's joint offset $d=1+\kappa g^*$ in place of $d^*$, the scale $g^*$ is
certified self-avoiding at every twist we tested between $51.5^\circ$ and
$73^\circ$ (in steps of $0.5^\circ$), and the certificate fails just outside
this window.  At the artwork's $45^\circ$, however, the ratio $1/\sqrt5$
produces clashes; the artwork's $g$ lies $2.8\%$ below it, on the safe side of
$g_c$.

\begin{figure}[t]
  \centering
  \includegraphics[width=0.495\linewidth]{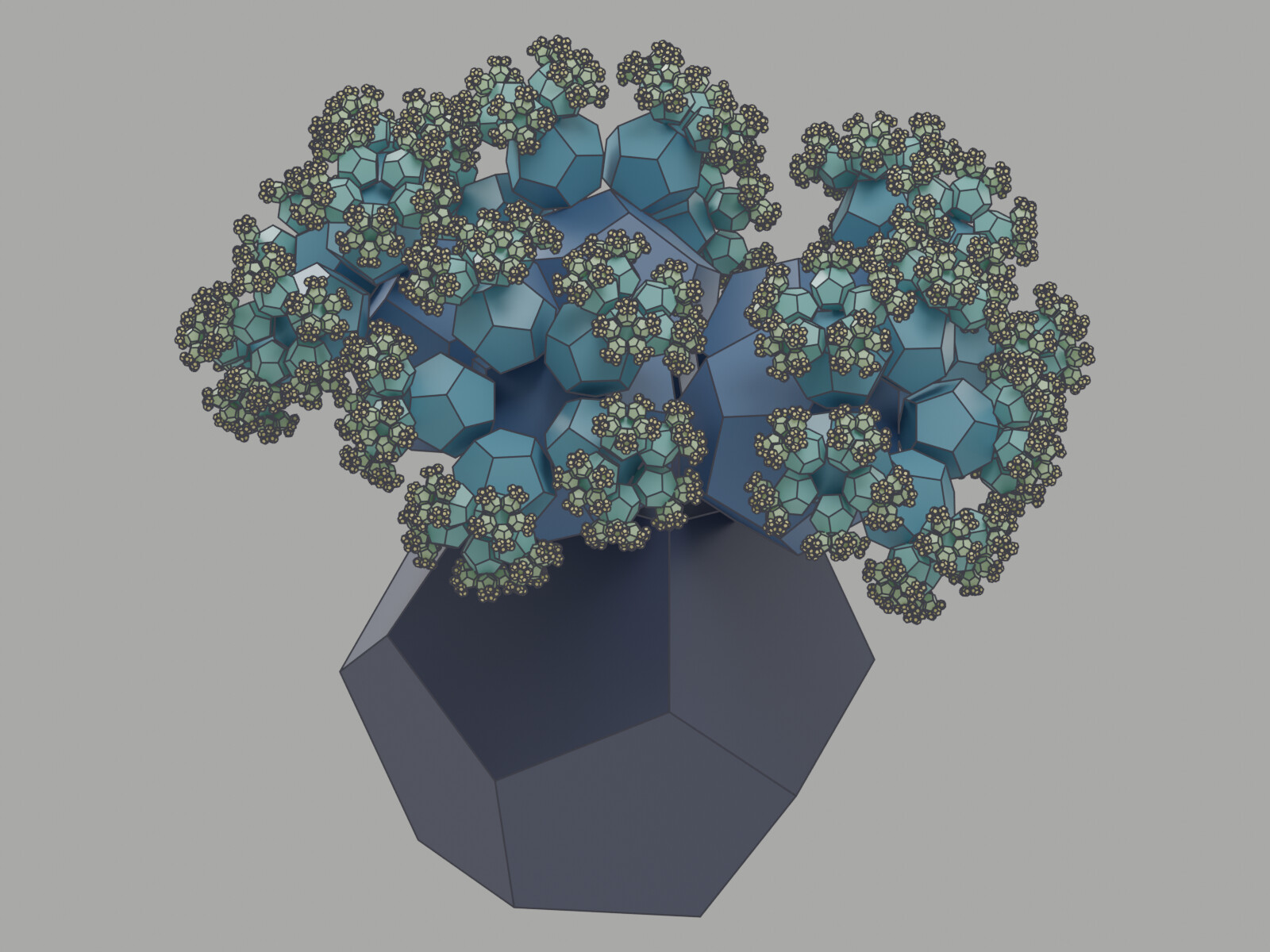}\hfill
  \includegraphics[width=0.495\linewidth]{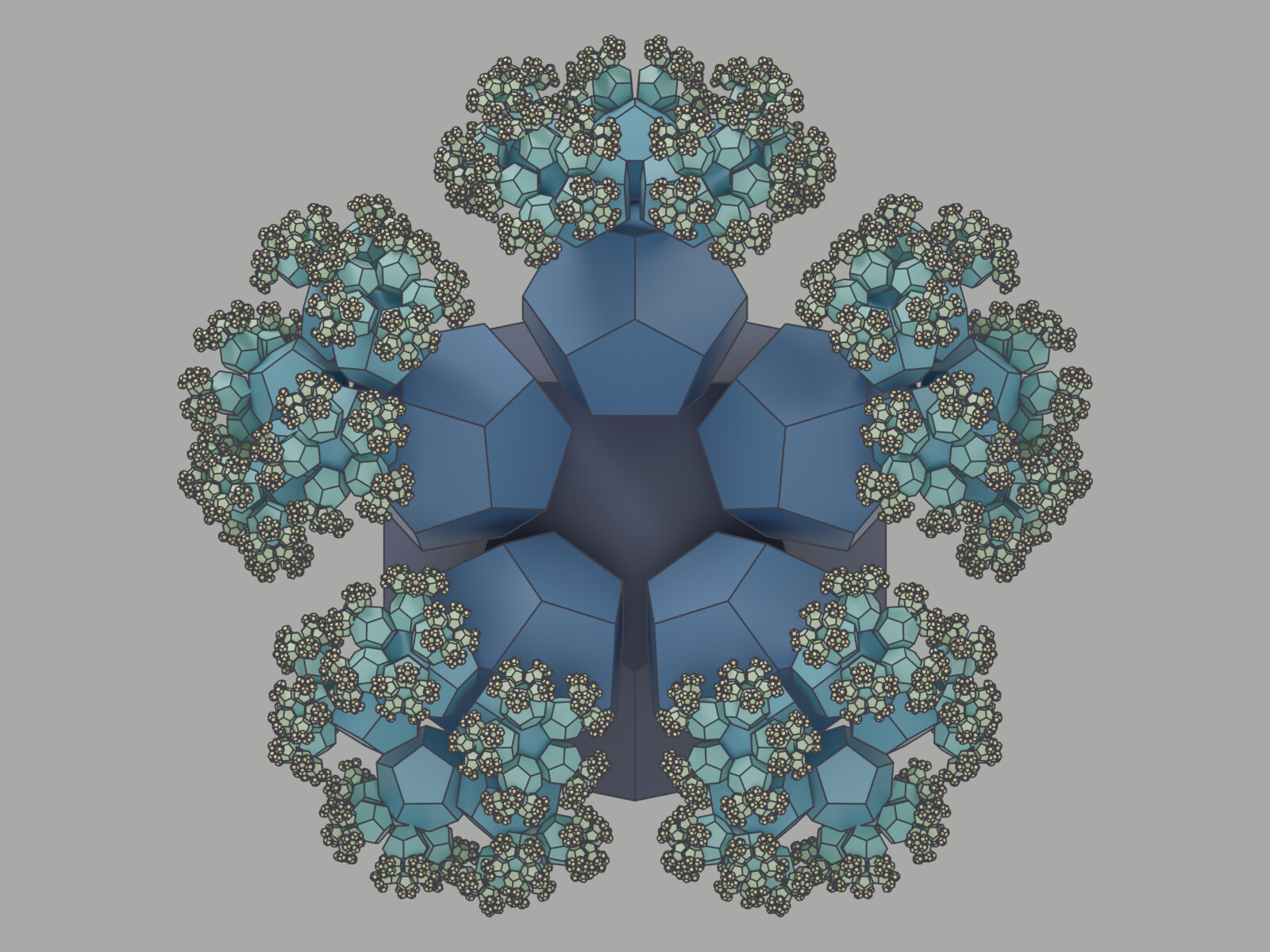}
  \caption{The Pythagorean dodecahedral tree $T^*$ ($g=1/\sqrt5$,
  $\theta=\arccos(1/\sqrt5)$, $d=1+1/\sqrt5$), depth $5$.  Every child stands
  on the plane of a lateral face of its parent and grows along that face's
  axis.}
  \label{fig:pyth}
\end{figure}

\section{Growing from every face}\label{sec:allfaces}

The artwork grows from one face only.  Could the tree grow from all twelve?
The answer depends on what ``growing from every face'' means; we consider three
readings.

\paragraph{Twelve crowns.}  Put a copy of the crown tree on every face,
$\bigcup_j\gamma_j(T)$ with $\gamma_j\in I$ mapping $F_0$ to $F_j$ (by
Proposition~\ref{prop:sym} the choice of $\gamma_j$ does not matter).  Every
edge lies on two faces, so it now carries two \emph{twin} children with the
same centre $d\,m$, twisted by $+\theta$ and $-\theta$ about the edge.  The
twins coincide when $\theta\in\{0^\circ,90^\circ\}$ modulo $180^\circ$, because
$\Rot(e,2\theta)$ is then a symmetry of $B$, and cross each other otherwise;
in both cases the tree clashes at depth $1$.  For the artwork's parameters
there are $90$ clashing pairs at depth $1$ and $6390$ by depth $4$.

\paragraph{One child per edge.}  The natural repair keeps a single child on
each of the $30$ edges.  This forces a choice of \emph{owner} for every edge,
one of its two faces.  The child is twisted by $+\theta$ about the edge
oriented counter-clockwise around its owner, and it carries the owner's
subtree.  Call two edges that meet at a vertex of a face $F$ a
\emph{positive} pair of $F$ if $F$ owns both, and a \emph{negative} pair if
$F$ owns neither.  A rotation taking $F$ to $F_0$ shows that the two children
of a positive pair are congruent to two adjacent crown siblings with twist
$\theta$, and those of a negative pair to siblings with twist $-\theta$,
equivalently $180^\circ-\theta$.

\begin{theorem}\label{thm:parity}
For every choice of owners there are at least six positive and at least six
negative pairs.  Both bounds are attained at once: there is an ownership with
exactly six pairs of each kind.
\end{theorem}

\begin{proof}
Let face $F_j$ own $o_j$ of its five edges.  Each edge has one owner, so
$\sum_j o_j=30$ and $\sum_j(5-o_j)=30$.  On the $5$-cycle of edges of a face,
$c$ marked edges always include at least $\max(0,c-2)$ adjacent pairs, since
at most two of five edges can be pairwise non-adjacent.  Hence there are at
least $\sum_j(o_j-2)=30-24=6$ positive pairs, and likewise
$\sum_j(5-o_j-2)=6$ negative ones.  An integer program over the $2^{30}$
ownership choices finds an assignment with exactly six negative pairs
(Figure~\ref{fig:parity}, right); every face owns two or three edges, and
counting shows that it has exactly six positive pairs as well.
\end{proof}

So any one-child-per-edge tree contains siblings twisted both ways, by
$\theta$ and by $180^\circ-\theta$.  For the artwork's $g$ and $d$, a pair of
adjacent siblings twisted coherently by $\theta$ is clash-free only for
$\theta\in[38.3^\circ,78.3^\circ]$ modulo $180^\circ$ (Figure~\ref{fig:parity},
left; computed on a $0.1^\circ$ grid), and this interval contains no $\theta$
together with $180^\circ-\theta$: the larger of the two overlaps never drops
below $0.058$ of the child's size.

\begin{corollary}
At the artwork's proportions, no one-child-per-edge extension to all twelve
faces is clash-free, whatever the twist and the owners.  At $\theta=45^\circ$,
optimal owners still leave exactly six clashing pairs at depth $1$, and deeper
generations add more ($48$ at depth $3$).
\end{corollary}

Smaller children escape the obstruction.  Figure~\ref{fig:phase} (bottom)
shows the certified region for the best owners.  At $\theta=45^\circ$ every
scale we tested up to $g=0.35$ (in steps of $0.0125$) is certified
self-avoiding at all depths, and the largest certified scale is
$g=\ALLMAXG$, at $\theta=\ALLMAXTH$.
Figure~\ref{fig:allfaces} shows both sides of the boundary.

\paragraph{Every node sprouts everywhere.}  The most literal reading lets every
node grow a child on each of its own $30$ edges.  Here a simple volume count
rules out the artwork's scale.

\begin{proposition}\label{prop:volume}
If a tree with $N$ maps of ratio $g$ is self-avoiding, then $Ng^3\le1$.
\end{proposition}

\begin{proof}
Nodes of the same depth are never parent and child, so the $N^n$ nodes of
depth $n$ have disjoint interiors.  They all lie in a bounded invariant set
$Q$ (a large enough ball will do), so $N^ng^{3n}\vol(B)\le\vol(Q)$ for every
$n$, which forces $Ng^3\le1$.
\end{proof}

For $N=30$ this requires $g\le30^{-1/3}\approx0.322$, whereas the artwork has
$30g^3\approx2.46$.  In fact every scale we tried, down to $g=0.2$, clashes by
depth $2$, and in each case some grandchildren grow back into their
grandparent.

\begin{figure}[t]
  \centering
  \includegraphics[width=0.49\linewidth]{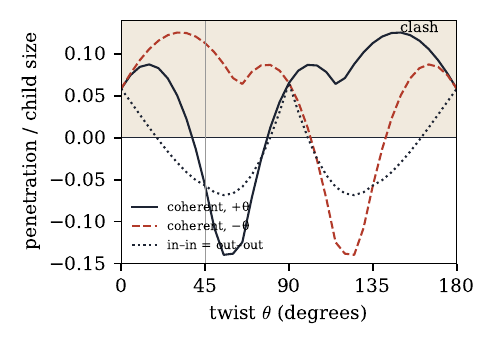}\hfill
  \includegraphics[width=0.40\linewidth]{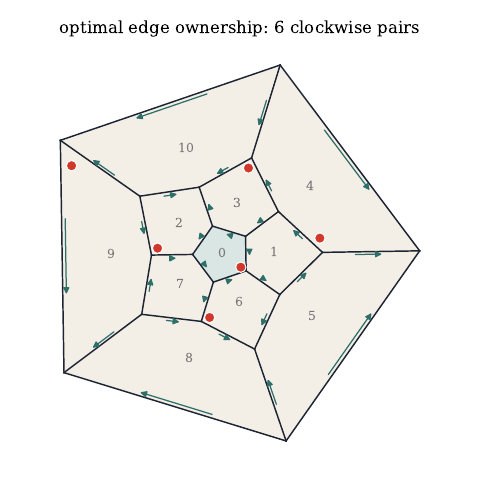}
  \caption{Left: overlap of two adjacent children (artwork's $g,d$) for the
  four relative orientations of their twists; values above zero are clashes.
  The two ``coherent'' curves are translates of each other by $180^\circ$, and
  no twist keeps both below zero.  Right: an optimal edge ownership on a
  Schlegel diagram (face $11$, opposite the crown, is outermost).  Arrows show
  the counter-clockwise orientation around the owner, and dots mark the six
  unavoidable negative pairs.}
  \label{fig:parity}
\end{figure}

\begin{figure}[t]
  \centering
  \includegraphics[width=0.49\linewidth]{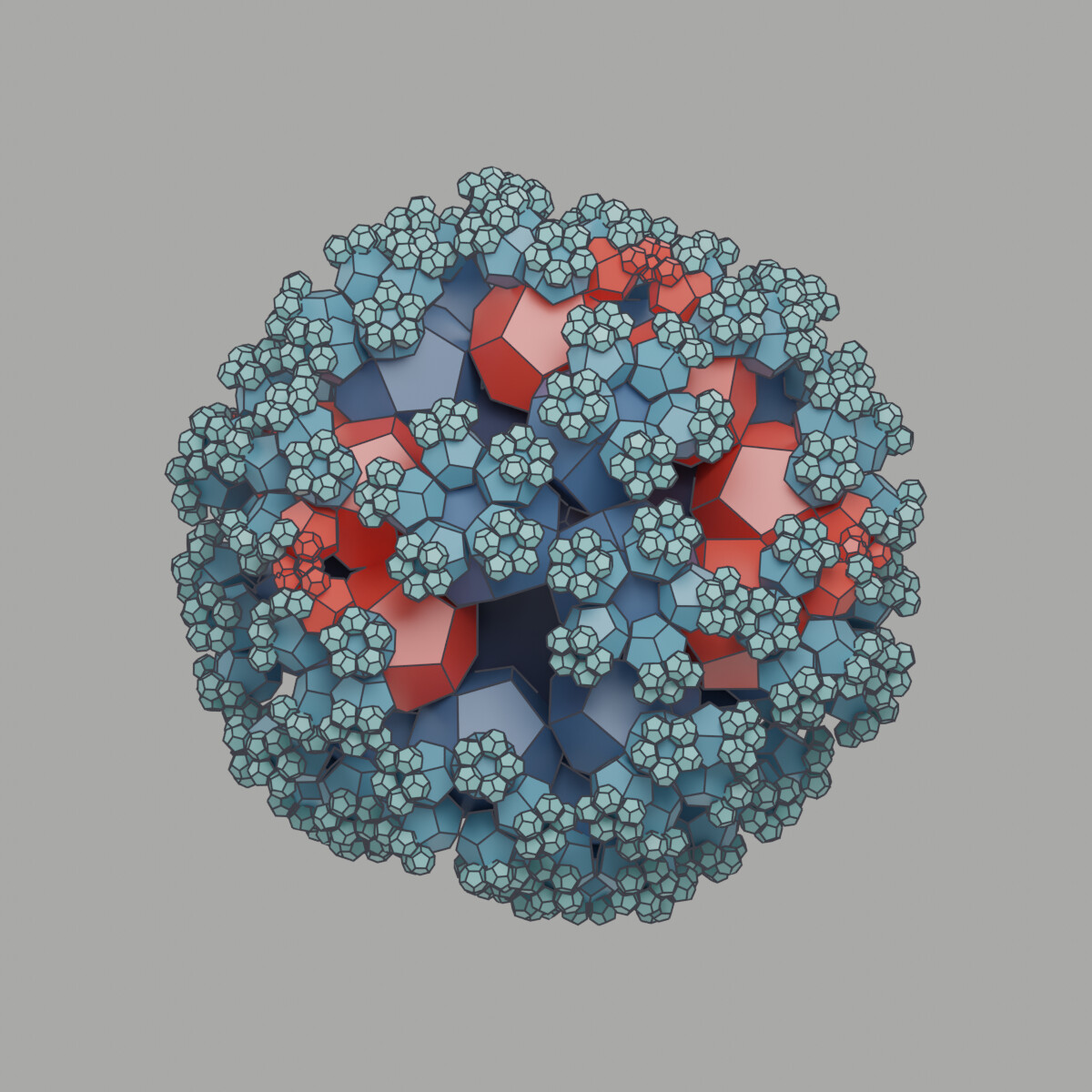}\hfill
  \includegraphics[width=0.49\linewidth]{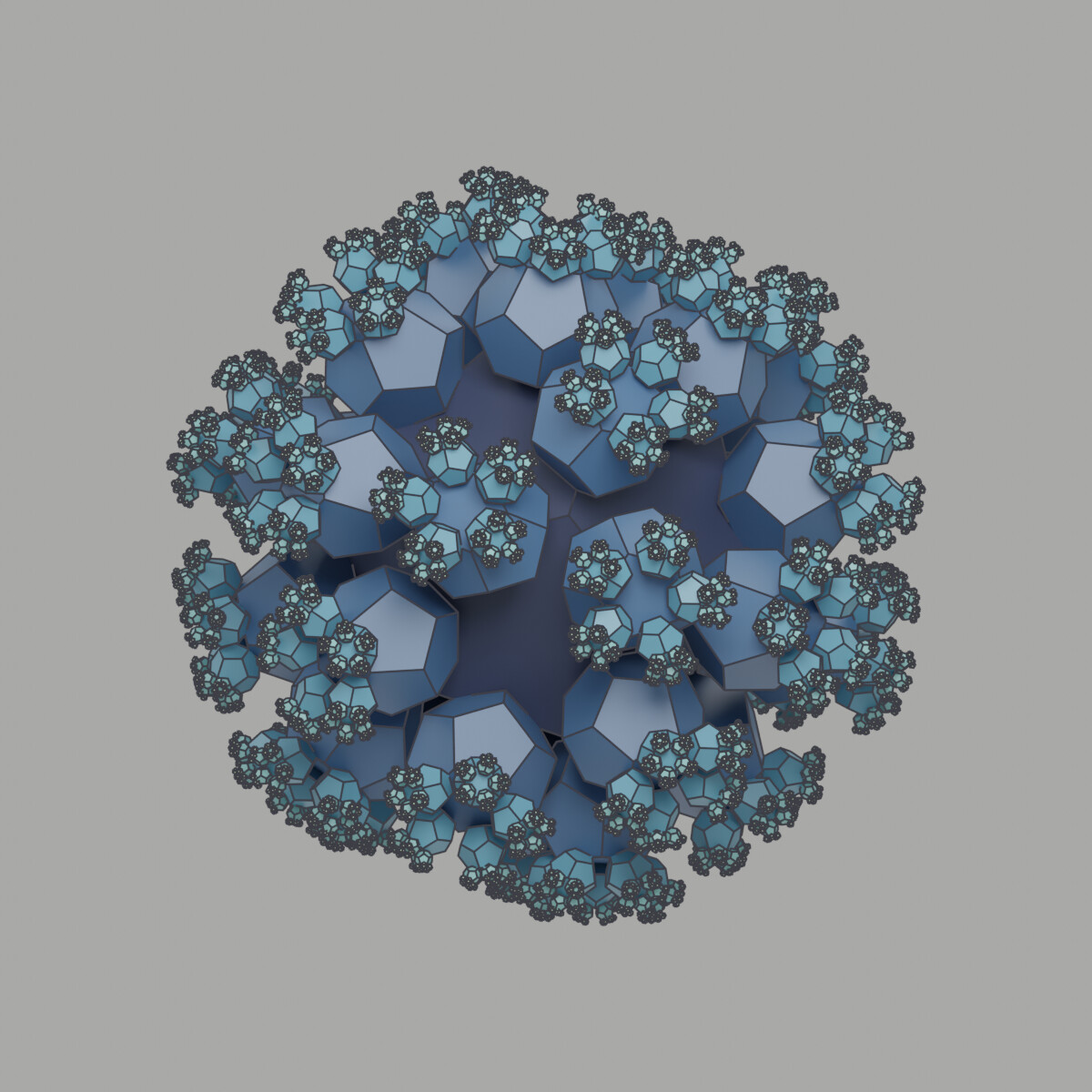}
  \caption{Growing from all twelve faces with optimal owners.  Left: the
  artwork's proportions ($g\approx0.435$, depth~3); the dodecahedra involved
  in clashes are red.  Right: $g=0.35$ (depth~4), certified self-avoiding.}
  \label{fig:allfaces}
\end{figure}

\section{Closing remarks}

The artwork reached its proportions by eye, through a factor $0.36$, a
$45^\circ$ twist and a constant called \texttt{gold} that is not the golden
ratio.  It sits just inside the region where the tree avoids itself, and
$2.8\%$ below the scale at which its surface area would be conserved from one
generation to the next.  Moving to that scale requires a larger twist,
$\theta^*=\arccos(1/\sqrt5)$, and then everything locks into place: children
rest flat on their parents' faces, the surface identity is the dodecahedron's
own Pythagorean theorem, and the crown becomes a set of dimension exactly two.

Some questions remain open.  The certificate decides individual parameter
points; the exact shape of the boundary curves in Figure~\ref{fig:phase} is
unknown, as is whether the kissing distance $1+g$ occurs only at the twists
we found.  Non-uniform twists, unequal scales, or children on vertices
instead of edges would all give new trees.  The same constructions could be
tried on the $120$-cell, whose cells are dodecahedra.

\section*{Acknowledgements}

I thank Eduardo Virue\~na Silva, professor at the Escuela Superior de F\'isica
y Matem\'aticas of the Instituto Polit\'ecnico Nacional
(\url{https://michelle.esfm.ipn.mx/eduardo.viruena/}).  His classes are the
reason I came to love programming and recreational mathematics, and his
lectures directly inspired this work.

\section*{Code and data}

All code and data are public at
\url{https://github.com/MiguelPerezI/dodecahedron-pythagoras-tree}: the Python
package \texttt{dodecatree} (geometry, the faithful port of the original
programs, the exact construction, collision tests and the certificate), the
original C\texttt{++} sources, and scripts that rebuild the original programs,
compare them with the Blender scene, run the clash study, and generate every
figure and number in this article.  STL and OBJ models of all trees shown can
be regenerated with \texttt{scripts/make\_tree.py}.

\end{document}